\documentclass[11pt, a4paper,reqno]{amsart}
\usepackage{caption}
\usepackage{graphicx}
\usepackage[utf8]{inputenc}
\usepackage{amsmath}
\usepackage{amssymb}
\usepackage{hyperref}
\usepackage{amsthm}
\usepackage{enumerate}
\usepackage{wasysym}

\usepackage{extarrows}
\usepackage{polynom}
\usepackage{dsfont}
\usepackage{verbatim}
\usepackage[shortlabels]{enumitem}
\usepackage[toc,page]{appendix}
\usepackage{a4wide}
\usepackage{todonotes}
\usepackage{pst-node}
\usepackage{ esint }
\usetikzlibrary{math}
\usetikzlibrary{decorations.pathreplacing}
\theoremstyle{definition}
\newtheorem{theorem}{Theorem}[section]

\newtheorem{lemma}[theorem]{Lemma}

\numberwithin{equation}{section}
\usepackage{tikz-3dplot}
\newcommand{\abs}[1]{\left\lvert#1\right\rvert}

\newcommand{\norm}[1]{\left\|#1\right\|}

\newcommand{\R}{\mathbb R}

\renewcommand{\epsilon}{\varepsilon}

\newcommand{\dx}{\, \mathrm d}

\renewcommand{\phi}{\varphi}
\usetikzlibrary{hobby}    

\usepackage{float}

\UseRawInputEncoding
\begin{document}
\allowdisplaybreaks

\title[A trajectorial interpretation of Moser's proof of the Harnack inequality]{A trajectorial interpretation of Moser's proof of the Harnack inequality}

\date{\today}
\author{Lukas Niebel$^*$}
\email[Lukas Niebel (corresponding author)]{lukas.niebel@uni-ulm.de}
\urladdr{https://lukasniebel.github.io}
\author{Rico Zacher}
\email[Rico Zacher]{rico.zacher@uni-ulm.de}
\urladdr{https://www.uni-ulm.de/mawi/iaa/members/zacher}

\address[Lukas Niebel, Rico Zacher]{Institut f\"ur Angewandte Analysis, Universit\"at Ulm, Helmholtzstra\ss{}e 18, 89081 Ulm, Germany.}
\thanks{\emph{Declarations of interests:} none.}

\maketitle
\begin{abstract}
{In 1971 Moser published a simplified version of his proof of the parabolic Harnack inequality. The core new ingredient is a fundamental lemma due to Bombieri and Giusti, which combines an $L^p-L^\infty$-estimate with a weak $L^1$-estimate for the logarithm of supersolutions. In this note, we give a novel proof of this weak $L^1$-estimate. The presented argument uses \textit{parabolic trajectories} and does not use any Poincar\'e inequality. Moreover, the proposed argument gives a geometric interpretation of Moser's result and could allow transferring Moser's method to other equations.}
\end{abstract}

\vspace{1em}
{\centering \textbf{Mathematics Subject Classification.} 35K10, 35B45, 35B65, 35J15. \par}
\vspace{1em}

\section{Moser's proof of the parabolic Harnack inequality}

 For $T>0$ and $\Omega \subset \R^d$ open we set $\Omega_T := (0,T) \times \Omega$. Let $A = A(t,x) \in L^\infty(\Omega_T;\R^{d \times d})$ be such that
\begin{enumerate}
	\item[(H1)] $\lambda \abs{\xi}^2 \le \langle A(t,x) \xi, \xi \rangle$ for all $\xi \in \R^d$ and almost all $(t,x) \in \Omega_T$,
	\item[(H2)] $\sum\limits_{i,j = 1}^d \abs{a_{ij}(t,x)}^2 \le \Lambda^2$ for almost all $(t,x) \in \Omega_T$,
\end{enumerate}
for some constants $0<\lambda <\Lambda$.
We are interested in solutions to the parabolic equation 
\begin{equation} \label{eq:par}
	\partial_t u = \nabla \cdot (A(t,x) \nabla u).
\end{equation}

We say $u \in C([0,T];L^2(\Omega)) \cap L^2((0,T);H^1(\Omega))$ is a weak (super-, sub-) solution to the parabolic equation \eqref{eq:par} in $\Omega_T$ if equality ($\ge $,$\le $) in \eqref{eq:par} holds in the distributional sense.  We set $\mu := \frac{1}{\lambda}+\Lambda$ if $A$ is symmetric and $\mu := (\frac{1}{\lambda}+\Lambda)^2$ otherwise. 

Next, we introduce parabolic cylinders. For $r>0$, $t_0 \in \R$ and $x_0 \in \R^d$, we define 
\begin{align*}
	Q_r^-(t_0,x_0) &= (t_0-r^2,t_0] \times B_r(x_0), \\
	Q_r^+(t_0,x_0) &= [t_0,t_0+r^2) \times B_r(x_0), \\
	Q_r(t_0,x_0) &= (t_0-r^2,t_0+r^2) \times B_r(x_0).
\end{align*}

With this notation at hand, the Harnack estimate can be stated as follows.
\begin{theorem} \label{thm:harnack}
	Let $\delta \in (0,1) $, $\tau >0$. There exists $C = C(d,\delta,\tau)>0$ such that for any cylinder $\tilde{Q} = (t_0-\delta \tau r^2,t_0+2 \tau r^2) \times B_{ r}(x_0) \subset \Omega_T$ with $r>0$ and any nonnegative weak solution $u$ of equation \eqref{eq:par} in $\tilde{Q}$ satisfies
	\begin{equation*}
		\sup_{Q_-} u \le C^\mu \inf_{Q_+} u
	\end{equation*}
	where $Q_- = (t_0,t_0+\delta \tau r^2) \times B_{\delta r}(x_0)$ and $Q_+ = (t_0+(2-\delta) \tau r^2,t_0+2\tau r^2) \times B_{\delta r}(x_0)$.
\end{theorem}

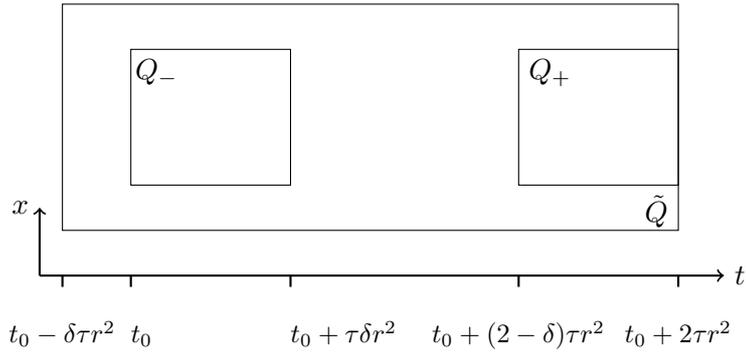
\begin{figure}[H]
\centering
\tikzmath{\x = 0.2; \y = 0.2; \offer = 0.5;}
	\begin{tikzpicture}[scale=3]
  \draw[thick,->] (0, 0) -- (2.5+\offer,0) node[right] {$t$};
  \draw[thick,->] (0, 0) -- (0,0.3) node[left] {$x$};
  \draw[draw=black] (0.3-\x,0.4-\y) rectangle ++(2.2+\offer,1);
  \draw[draw=black] (1.8+\offer-\x,0.6-\y) rectangle ++(0.7,0.6);
  \draw (1.8+\offer-\x,1.2-\y) node[anchor=north west] {$Q_+$};
  \draw[draw=black] (0.6-\x,0.6-\y) rectangle ++(0.7,0.6);
  \draw (0.85-\x,1.2-\y) node[anchor=north east] {$Q_-$};
  \draw (2.5-\x+\offer,0.6-\y) node[anchor=north east] {$\tilde{Q}$};
  \draw [thick] (0.3-\x, 0) -- ++(0, -.05) ++(0, -.15) node [below, outer sep=0pt, inner sep=0pt] {\small\(t_0-\delta\tau r^2\)};
  \draw [thick] (0.6-\x, 0) -- ++(0, -.05) ++(0, -.15) node [below right, outer sep=0pt, inner sep=0pt] {\small\(t_0 \vphantom{r^2}\)};
  \draw [thick] (1.3-\x, 0) -- ++(0, -.05) ++(0, -.15) node [below right, outer sep=0pt, inner sep=0pt] {\small\(t_0+\tau\delta r^2\)};
  \draw [thick] (1.8-\x+\offer, 0) -- ++(0, -.05) ++(0, -.15) node [below, outer sep=0pt, inner sep=0pt] {\small\(t_0+(2-\delta)\tau r^2\)};
  \draw [thick] (2.5-\x+\offer, 0) -- ++(0, -.05) ++(0, -.15) node [below, outer sep=0pt, inner sep=0pt] {\small\(t_0+2\tau r^2\)};
 \end{tikzpicture}
 \caption{The sets $\tilde{Q},Q_-,Q_+$ in the Harnack inequality of Theorem \ref{thm:harnack}.}
\end{figure}

For the parabolic problem \eqref{eq:par}, the Harnack inequality was proven first in \cite{moser_harnack_1964,moser_correction_1967}. A few years later, Moser gave a new proof of the Harnack inequality in \cite{moser_pointwise_1971}, which avoids the technical approach via $\mathrm{BMO}$ functions. We sketch the refined proof, compare also \cite{bonforte_explicit_2020,saloff_aspects_2001} for a modern exposition. The Harnack inequality is a strong \textit{a priori} estimate and plays a crucial role in the regularity theory of partial differential equations. 

The first ingredient of Moser's proof is an $L^p-L^\infty$ estimate for small $p$. 

\begin{lemma}\label{lem:lplinf}
	Let $\delta \in (0,1)$, $\delta \le r < R \le 1$ and $t_0 \in (0,T)$, $x_0 \in \Omega$ with $Q_R(t_0,x_0) \subset \Omega_T$. There exists a constant $c_1 = c_1(d,\delta)$ such that for any positive weak solution $u$ to equation \eqref{eq:par} in $\Omega_T$ we have\begin{align*}
\sup _{Q_r(t_0,x_0)} u^p &\leq \frac{c_1}{(R-r)^{d+2}} \iint\limits_{Q_R(t_0,x_0)}  u^p \dx x \dx t, \quad\quad \forall p \in\left(0, \frac{1}{\mu}\right), \\
\sup _{Q_r^{-}(t_0,x_0)} u^p &\leq \frac{c_1}{(R-r)^{d+2}} \iint\limits_{Q_R^{-}(t_0,x_0)} u^p \dx x \dx t, \quad\quad \forall p \in\left(-\frac{1}{\mu}, 0\right).
\end{align*}
\end{lemma}

\begin{proof}
	The proof can be found in \cite{bonforte_explicit_2020,moser_harnack_1964,saloff_aspects_2001}.
\end{proof}
\begin{figure}[H]
\begin{minipage}{.5\textwidth}
\tikzmath{\x = 0.2; \y = 0.2;}
	\begin{tikzpicture}[scale=3]
  \draw[thick,->] (0, 0) -- (0.3,0) node[right] {$t$};
  \draw[thick,->] (0, 0) -- (0,0.3) node[left] {$x$};
  \draw[draw=black] (0.4-\x,0.4-\y) rectangle ++(1.5,1);
  \draw[draw=black] (0.75-\x,0.6-\y) rectangle ++(0.8,0.6);
  \draw (0.4-\x,1.4-\y) node[anchor=north west] {\tiny $Q_R(t_0,x_0)$};
  \draw (0.75-\x,1.2-\y) node[anchor=north west] {\tiny $Q_r(t_0,x_0)$};
  \filldraw (1.15-\x,0.9-\y) circle[radius=0.4pt];
  \draw (1.15-\x,0.9-\y) node[anchor= west] {\tiny $(t_0,x_0)$};
 \end{tikzpicture}
\end{minipage}%
\begin{minipage}{.5\textwidth}
\tikzmath{\x = 0.2; \y = 0.2;}
	\begin{tikzpicture}[scale=3]
  \draw[thick,->] (0, 0) -- (0.3,0) node[right] {$t$};
  \draw[thick,->] (0, 0) -- (0,0.3) node[left] {$x$};
  \draw[draw=black] (0.4-\x,0.4-\y) rectangle ++(1.5,1);
  \draw[draw=black] (1.1-\x,0.6-\y) rectangle ++(0.8,0.6);
  \draw (0.4-\x,1.4-\y) node[anchor=north west] {\tiny $Q_R^-(t_0,x_0)$};
  \draw (1.1-\x,1.2-\y) node[anchor=north west] {\tiny $Q_r^-(t_0,x_0)$};
  \filldraw (1.9-\x,0.9-\y) circle[radius=0.4pt];
  \draw (1.9-\x,0.9-\y) node[anchor= west] {\tiny $(t_0,x_0)$};
 \end{tikzpicture}
	
\end{minipage}
\caption{The cylinders in Lemma \ref{lem:lplinf}.}
\end{figure}
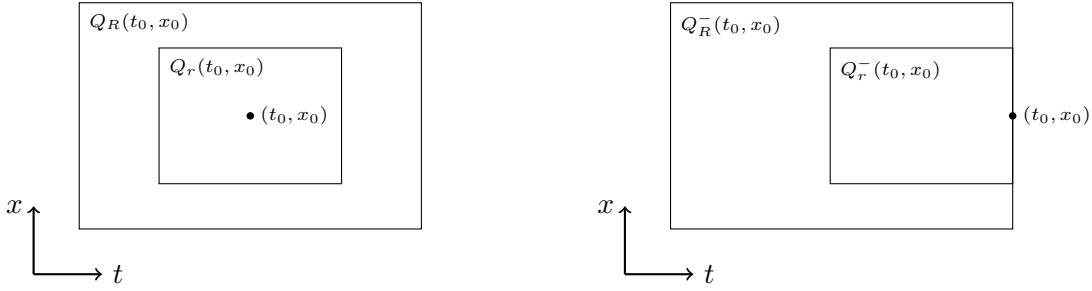

The second ingredient is a weak $L^1$-estimate for the logarithm of supersolutions. 
\begin{lemma} \label{lem:weakl1log}
	Let $\delta, \eta \in (0,1)$ and $\epsilon,\tau >0$. Then for any $t_0 \ge 0$, $r>0$ with $t_0+\tau r^2\le T$, any ball $B = B_r(x_0) \subset \Omega$, and any weak supersolution $u \ge \epsilon>0$ of equation \eqref{eq:par} in $(t_0,t_0+\tau r^2) \times B$, there is a constant $c = c(u)$ such that 
	\begin{align*}
		&\abs{\{ (t,x) \in K_- \colon \log u(t,x) > c+s \}} \le C \mu r^2 \abs{B}s^{-1}, \quad s>0, \\
		&\abs{\{ (t,x) \in K_+ \colon \log u(t,x) < c-s \}} \le C \mu r^2 \abs{B}s^{-1}, \quad s>0, \\
	\end{align*}
	where $K_- = (t_0,t_0+\eta \tau r^2) \times \delta B$, $K_+ = (t_0+\eta \tau r^2, t_0+\tau r^2) \times \delta B$ and $C = C(d,\delta,\eta,\tau)$.
\end{lemma}

\begin{proof}
	The result is due to \cite{moser_harnack_1964,moser_correction_1967,moser_pointwise_1971}. A detailed exposition can be found in \cite{bonforte_explicit_2020} for symmetric $A$. The proof also extends to the general case. However, one needs to replace $(1/\lambda+\Lambda)$ by $(1/\lambda+\Lambda)^2$.
\end{proof}

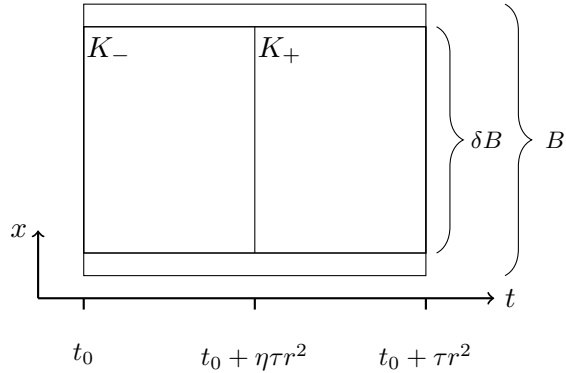
\begin{figure}[H]
\centering
\tikzmath{\x = 0.2; \y = 0.2;}
\begin{tikzpicture}[scale=3]
  \draw (1.15-\x,0.4-\y) -- (1.15-\x,1.4-\y);
  \draw[thick,->] (0, 0) -- (2,0) node[right] {$t$};
  \draw[thick,->] (0, 0) -- (0,0.3) node[left] {$x$};
  \draw[draw=black] (0.4-\x,0.3-\y) rectangle ++(1.5,1.2);
  \draw[draw=black] (0.4-\x,0.4-\y) rectangle ++(1.5,1);
  \draw[draw=black] (0.4-\x,0.4-\y) rectangle ++(1.5,1);
  \draw (0.65-\x,1.4-\y) node[anchor=north east] {$K_-$};
  \draw (1.4-\x,1.4-\y) node[anchor=north east] {$K_+$}; 
  \draw [thick] (0.4-\x, 0) -- ++(0, -.05) ++(0, -.15) node [below, outer sep=0pt, inner sep=0pt] {\small\(t_0\)};
  \draw [thick] (1.15-\x, 0) -- ++(0, -.05) ++(0, -.15) node [below, outer sep=0pt, inner sep=0pt] {\small\(t_0+\eta\tau r^2 \)};
  \draw [thick] (1.9-\x, 0) -- ++(0, -.05) ++(0, -.15) node [below, outer sep=0pt, inner sep=0pt] {\small\(t_0+\tau r^2 \)};
\draw [decorate,decoration={brace,amplitude=10pt,mirror,raise=4pt},yshift=0pt]
(2.2-\x,0.3-\y) -- (2.2-\x,1.5-\y) node [black,midway,xshift=0.8cm] {\footnotesize
$B$};
\draw [decorate,decoration={brace,amplitude=10pt,mirror,raise=4pt},yshift=0pt]
(1.9-\x,0.4-\y) -- (1.9-\x,1.4-\y) node [black,midway,xshift=0.8cm] {\footnotesize
$\delta B$};
\end{tikzpicture}
\caption{The cylinders $K_-,K_+$ and the spatial balls $\delta B,B$ in Lemma \ref{lem:weakl1log}.}
\end{figure}

The information of Lemma \ref{lem:lplinf} and Lemma \ref{lem:weakl1log} is combined using the following lemma, which is originally due to Bombieri and Giusti \cite{bombieri_harnack_1972}.

\begin{lemma} \label{lem:bomb}
	 Let $U_\sigma$, $0<\sigma \le 1$ be a collection of measurable subsets of a fixed finite measure space endowed with a measure $\nu$ such that $U_{\sigma '} \subset U_\sigma$ if $\sigma' \le \sigma$. Furthermore, let $C_1,C_2>0$, $\delta,\eta \in (0,1)$, $\tilde{\mu} > 1$, $\gamma>0$. Suppose that a positive measurable function $f \colon U_1 \to \R$ satisfies the following two conditions:
	\begin{enumerate}
		\item for all $0<\delta \le r < R \le 1$ and $0< p < 1/\tilde{\mu}$ we have \begin{equation*}
			\sup_{U_{r}} f^p \le \frac{C_1}{(R-r)^\gamma \nu(U_1)} \int_{U_R}f^p \dx \nu,
		\end{equation*}
		\item $\nu(\{ \log f > s \}) \le C_2\tilde{\mu} \ \nu(U_1) s^{-1} \text{ for all } s>0 $.
	\end{enumerate}
	Then, 
	\begin{equation*}
		\sup_{U_\delta } f \le \exp\left( \left[2C_2 + \frac{8C_1^3}{(1-\delta)^{2\gamma}} \right] \tilde{\mu} \right).
	\end{equation*}
\end{lemma}

\begin{proof}
	The proof can be found in \cite{bonforte_explicit_2020,clement_apriori_2004,saloff_aspects_2001,zacher_harnack_2013}. 	
\end{proof}

With these three ingredients at hand, we can now turn to the proof of the Harnack inequality. We decided to revisit the proof to illustrate the difference in our method.

\begin{proof}[Proof of Theorem \ref{thm:harnack}]

Employing the translation $(t,x) \mapsto (t-t_0,x-x_0)$ and the scaling $r \mapsto (r^2 t, rx)$ we can reduce to the case $(t_0,x_0) = 0$ and $r = 1$. Moreover, by adjusting the parameters suitably, we may assume $(-(1-\delta),2\tau) \times B_\beta \subset \Omega_T$ for some $\beta>1$. Finally, we may assume $u \ge \epsilon$ for some $\epsilon>0$.

We aim to apply Lemma \ref{lem:bomb} twice, first to $f_1 = u^{-1}\exp(c(u))$ and then to $f_2 = u \exp(-c(u))$ on appropriately defined cylinders, where $c(u)$ is the constant given by Lemma \ref{lem:weakl1log}.

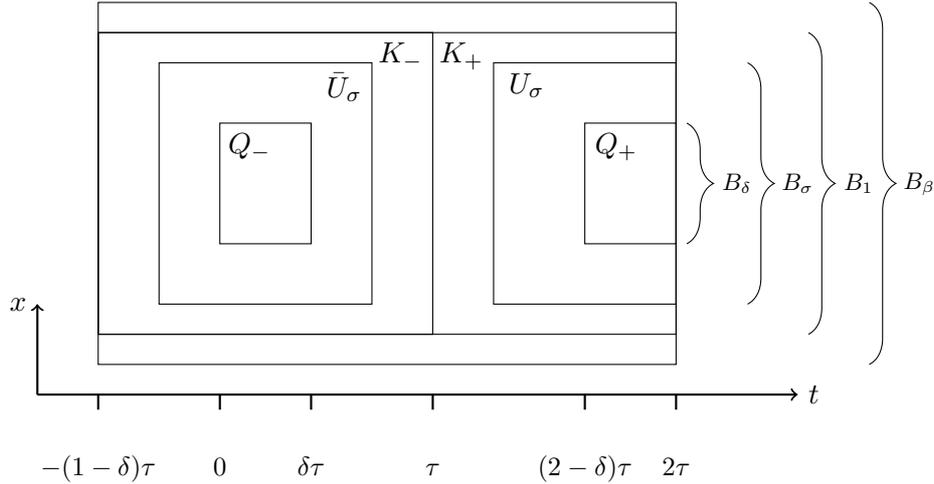
\begin{figure}[H] 
\centering
\tikzmath{\x = -0.2; \y = 0.2; \z = 0.1;}
\begin{tikzpicture}[scale=4]
  \draw[thick,->] (0, 0) -- (2.5,0) node[right] {$t$};
  \draw[thick,->] (0, 0) -- (0,0.3) node[left] {$x$};
  \draw[draw=black] (0.0-\x,0.3-\y) rectangle ++(1.9,1.2);
  \draw[draw=black] (0.0-\x,0.4-\y) rectangle ++(1.9,1);
  \draw[draw=black] (1.6-\x,0.7-\y) rectangle ++(0.3,0.4);
  \draw (1.6-\x,1.1-\y) node[anchor=north west] {$Q_+$};
  \draw[draw=black] (0.5-\x-\z,0.7-\y) rectangle ++(0.3,0.4);
  \draw (0.7-\x-\z,1.1-\y) node[anchor=north east] {$Q_-$};
  \draw[draw=black] (0.0-\x,0.4-\y) rectangle ++(1.1,1);
  \draw[draw=black] (0.2-\x,0.5-\y) rectangle ++(0.7,0.8);
  \draw[draw=black] (1.3-\x,0.5-\y) rectangle ++(0.6,0.8);
  \draw (1.1-\x,1.4-\y) node[anchor=north east] {$K_-$};
  \draw (1.3-\x,1.4-\y) node[anchor=north east] {$K_+$}; 
  \draw (0.9-\x,1.3-\y) node[anchor=north east] {$\bar{U}_\sigma$};
  \draw (1.5-\x,1.3-\y) node[anchor=north east] {$U_\sigma$};
  \draw [thick] (0.0-\x, 0) -- ++(0, -.05) ++(0, -.15) node [below, outer sep=0pt, inner sep=0pt] {\small\(-(1-\delta)\tau \) };
  \draw [thick] (0.5-\x-\z, 0) -- ++(0, -.05) ++(0, -.15) node [below, outer sep=0pt, inner sep=0pt] {\small\(0\)\vphantom{\small\(-(1-\delta)\tau \)}};
   \draw [thick] (0.8-\x-\z, 0) -- ++(0, -.05) ++(0, -.15) node [below, outer sep=0pt, inner sep=0pt] {\small\(\delta \tau \)\vphantom{\small\(-(1-\delta)\tau \)}};
    \draw [thick] (1.1-\x, 0) -- ++(0, -.05) ++(0, -.15) node [below, outer sep=0pt, inner sep=0pt] {\small\(\tau \)\vphantom{\small\(-(1-\delta)\tau \)}};
      \draw [thick] (1.6-\x, 0) -- ++(0, -.05) ++(0, -.15) node [below, outer sep=0pt, inner sep=0pt] {\small\( (2-\delta)\tau  \)};
      \draw [thick] (1.9-\x, 0) -- ++(0, -.05) ++(0, -.15) node [below, outer sep=0pt, inner sep=0pt] {\small\(2 \tau \)\vphantom{\small\(-(1-\delta)\tau \)}};
  \draw [decorate,decoration={brace,amplitude=10pt,mirror,raise=4pt},yshift=0pt]
(1.9-\x,0.7-\y) -- (1.9-\x,1.1-\y) node [black,midway,xshift=0.8cm] {\footnotesize
${B}_{\delta}$};
\draw [decorate,decoration={brace,amplitude=10pt,mirror,raise=4pt},yshift=0pt]
(2.1-\x,0.5-\y) -- (2.1-\x,1.3-\y) node [black,midway,xshift=0.8cm] {\footnotesize
${B}_{\sigma}$};
\draw [decorate,decoration={brace,amplitude=10pt,mirror,raise=4pt},yshift=0pt]
(2.3-\x,0.4-\y) -- (2.3-\x,1.4-\y) node [black,midway,xshift=0.8cm] {\footnotesize
${B}_{1}$};
\draw [decorate,decoration={brace,amplitude=10pt,mirror,raise=4pt},yshift=0pt]
(2.5-\x,0.3-\y) -- (2.5-\x,1.5-\y) node [black,midway,xshift=0.8cm] {\footnotesize
$B_\beta$};
\end{tikzpicture}
\caption{The cylinders $Q_-,Q_+,\bar{U}_\sigma,U_\sigma,K_-,K_+$ and the spatial balls $B_\delta,B_\sigma,B_1,B_\beta$ in the proof of Theorem \ref{thm:harnack} for $(t_0,x_0) = 0$ and $r = 1$. \label{fig:harnack}}
\end{figure}

Let us define the families of cylinders, illustrated in Figure \ref{fig:harnack}. For $0<\delta \le \sigma\le 1$ we set $U_\sigma = ((2-\sigma)\tau,2\tau) \times B_\sigma(0) $ and $ \bar{U}_\sigma = ( -(\sigma-\delta)\tau,\sigma \tau) \times B_\sigma(0) $. Set $K_- = \bar{U}_1$, $Q_- = \bar{U_\delta}$, $K_+ = U_1$ and $Q_+ = U_\delta$. 

 Next, with a slight abuse of notation, we apply Lemma \ref{lem:weakl1log} with $t_0 = -(1-\delta)$, $x_0 = 0$, $\tau = 2\tau+(1-\delta)\tau  \frac{1}{\beta^2} = (3-\delta)\tau \frac{1}{\beta^2} $, $\eta = \frac{2-\delta}{3-\delta}$, $r= \beta$ and $\delta = 1/\beta$. This gives the following estimate for the logarithm of $u$ on $K_-$ and $K_+$
 \begin{align}
		&\abs{\{ (t,x) \in K_- \colon \log u(t,x) > c+s \}} \le c_1 \mu  r^2 \abs{B}s^{-1}, \quad s>0, \label{eq:har1}\\
		&\abs{\{ (t,x) \in K_+ \colon \log u(t,x) < c-s \}} \le c_1 \mu  r^2 \abs{B}s^{-1}, \quad s>0,\label{eq:har2}
	\end{align}
	where $c_1 = c_1(d,\delta,\tau)$ and $c = c(u)$.

We consider $f_1 = u^{-1}\exp(c(u))$ on the sets $U_\sigma$. Lemma \ref{lem:lplinf} gives
\begin{equation*}
	\sup_{{U}_\sigma} f_1^p \le   \frac{ c_2\abs{{U}_1}^{-1}}{(\sigma'-\sigma)^{d+2}} \iint_{{U}_{\sigma'}} f_1^p \dx x \dx t, \quad \delta \le \sigma < \sigma'\le1 
\end{equation*}
for some constant $c_2 = c_2(d,\delta)>0$ and all $p \in (0,1/\mu)$. This, together with the estimate \eqref{eq:har2}, allows applying the lemma of Bombieri and Giusti to deduce 
\begin{equation*}
e^{c(u)} \le \exp\left(C\mu \right) \inf_{Q_+} u.
\end{equation*} 
with $C = C(d,\delta,\tau)$. 

Next, we consider $f_2 = u \exp(-c(u))$ on the sets $\bar{U}_\sigma$. Clearly, by the estimate in \eqref{eq:har1}, the second assumption of the lemma of Bombieri and Giusti is satisfied for $f_2$. Concerning the first assumption of the lemma of Bombieri and Giusti, we employ the estimates of Lemma \ref{lem:lplinf}. We have 
\begin{equation*}
	\sup_{\bar{U}_\sigma} f_2^p \le   \frac{ c_2\abs{\bar{U}_1}^{-1}}{(\sigma'-\sigma)^{d+2}} \iint_{\bar{U}_{\sigma'}} f_2^p \dx x \dx t, \quad \delta \le \sigma < \sigma'\le1 
\end{equation*}
for $p \in (0,1/\mu)$, which shows that the first assumption is satisfied. Consequently,
\begin{equation*}
	\sup_{Q_-} u \le e^{c(u)} \exp\left(C\mu \right)
\end{equation*}
for some $C = C(d,\delta,\tau)$ independent of $\mu $. Together we deduce
\begin{equation*}
	\sup_{Q_-} u \le C^\mu \inf_{Q_+} u.
\end{equation*}
for some constant $C = C(d,\delta,\tau)$.

\end{proof}

\section{Employing parabolic trajectories to prove a weak $L^1$-estimate for the logarithm of supersolutions}



To the authors' knowledge, the strategy used in Moser's work is the only known technique available to prove a weak $L^1$-estimate of the type in Lemma \ref{lem:weakl1log}. The method has been used in different contexts, such as parabolic non-local problems \cite{kassmann_local_2013}, time-fractional equations \cite{zacher_harnack_2013}, and many more \cite{albritton_regularity_2021,delmotte_parabolic_1999,lu_horm_1992}. However, in every one of those works, a particular product structure of the problem is used, detaching the time propagation from the diffusive part of the equation. A Poincar\'e inequality gives some control of the spatial part, which is then propagated in time by an intricate argument. 

We propose a novel way to prove a version of Lemma \ref{lem:weakl1log}, see Lemma \ref{lem:weakl1log2}. We do not use any Poincar\'e inequality in the spatial variable but compare the function directly at two points in space-time by following a suitable trajectory. The obtained result is slightly weaker because we need a positive gap between $K_-$ and $K_+$. However, it is still strong enough to prove the Harnack inequality of Theorem \ref{thm:harnack}.

\begin{lemma} \label{lem:weakl1log2}
	Let $\delta, \eta \in (0,1)$, $\iota < \min \{ \eta, 1-\eta \}$ and $\epsilon,\tau>0$. Then, for any $t_0 \ge 0$ and $r>0$ with $t_0+\tau r^2 \le T$ any ball $B = B_r(x_0) \subset \Omega$ and any weak supersolution $u \ge \epsilon>0$ of \eqref{eq:par} in $(t_0,t_0+\tau r^2) \times B$, there is a constant $c = c(u)$ such that 
	\begin{equation*}
		\abs{\{ (t,x) \in K_- \colon \log u (t,x) > c+s \}} \le C\mu r^2\abs{B}s^{-1}, \quad s>0,
	\end{equation*}
	and
	\begin{equation*}
		\abs{ \{ (t,x) \in K_+ \colon \log u (t,x) < c-s \}} \le C\mu r^2\abs{B}s^{-1}, \quad s>0,
	\end{equation*}
	where $K_- = (t_0,t_0+(\eta-\iota) \tau r^2) \times \delta B$, $K_+ = (t_0+(\eta+\iota) \tau r^2,t_0+\tau r^2) \times \delta B$ and $C = C(d,\delta,\eta,\iota,\tau)$.
\end{lemma}

\begin{proof}
	Consider $\tau = 1$ for simplicity and assume $r = 1$ and $(t_0,x_0) = 0$ by scaling and translation. Let us first assume that $A$ is symmetric. We write $B = B_1(0)$ and $\delta B = B_{\delta}(0)$. Let $u\ge \epsilon >0$ be a weak supersolution to equation \eqref{eq:par}, then $g = \log(u)$ is a weak supersolution to 
	\begin{equation*}
		\partial_t g = \nabla \cdot (A \nabla g) + \langle A \nabla g, \nabla g \rangle.
	\end{equation*}
	Let $\varphi \in C_c^\infty(\R^d;[0,1])$ be such that $\varphi = 1$ in $\delta B$ and $\varphi = 0$ in $B^c$ with $\norm{\nabla \varphi}_\infty \le \frac{2}{1-\delta}$. We set $c_\varphi = \int_{B}\varphi^2(x) \dx x$ and
	\begin{equation*}
		c(u) =\frac{1}{c_\varphi} \int_{B} g(\eta,y) \varphi^2(y) \dx  y.
	\end{equation*}
	Let $(t,x) \in K_-$, $y \in B$ and define $\gamma(r) = (\gamma_t(r),\gamma_x(r)) = (t+r^2(\eta-t),x+r(y-x))$, then $\gamma(r) \in (0,1) \times B$ for all $r \in [0,1]$.  This \textit{parabolic trajectory} is connected to the geometric properties of the PDE. The vector fields $\partial_t$ and $\partial_{x_i}$ together with the underlying scaling, determine the \textit{parabolic trajectories} we use to prove Lemma \ref{lem:weakl1log2}. The cylinders and this trajectory are sketched in Figure \ref{fig:log2}. The subscript $t$ also refers to the first component and the subscript $x$ also refers to the last $d$ components of $\gamma$.
	
	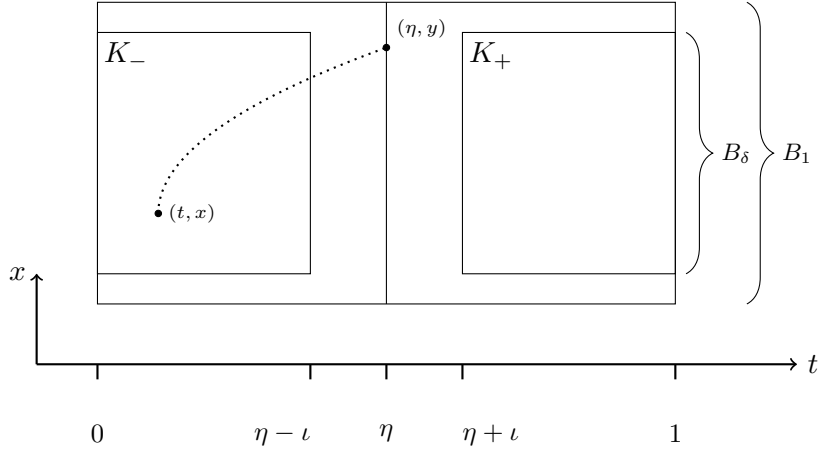
\begin{figure}[H]
\centering
\tikzmath{\x = 0.2; \y = 0.2;}
\begin{tikzpicture}[scale=4]
  \draw (1.35-\x,0.4-\y) -- (1.35-\x,1.4-\y);
  \draw[thick,->] (0, 0) -- (2.5,0) node[right] {$t$};
  \draw[thick,->] (0, 0) -- (0,0.3) node[left] {$x$};
  \draw[draw=black] (0.4-\x,0.4-\y) rectangle ++(1.9,1);
  \filldraw (0.2+1/5,0.5) circle[radius=0.3pt];
  \draw (0.2+1/5,0.5) node[anchor= west] {\tiny $(t,x)$};
  \filldraw (1.15,1.05) circle[radius=0.3pt];
  \draw (1.15,1.05) node[anchor= south west] { \tiny $(\eta,y)$};
  \draw[draw=black] (0.4-\x,0.5-\y) rectangle ++(0.7,0.8);
  \draw[draw=black] (1.6-\x,0.5-\y) rectangle ++(0.7,0.8);
  \draw (0.6-\x,1.3-\y) node[anchor=north east] {$K_-$};
  \draw (1.8-\x,1.3-\y) node[anchor=north east] {$K_+$}; 
  \draw [thick] (0.4-\x, 0) -- ++(0, -.05) ++(0, -.15) node [below, outer sep=0pt, inner sep=0pt] {\small\(0\)};
  \draw [thick] (1.6-\x, 0) -- ++(0, -.05) ++(0, -.15) node [below right, outer sep=0pt, inner sep=0pt] {\small\(\eta+\iota \)};
  \draw [thick] (1.35-\x, 0) -- ++(0, -.05) ++(0, -.15) node [below, outer sep=0pt, inner sep=0pt] {\small\(\eta \)};
  \draw [thick] (1.1-\x, 0) -- ++(0, -.05) ++(0, -.15) node [below left, outer sep=0pt, inner sep=0pt] {\small\(\eta-\iota \)};
  \draw [thick] (2.3-\x, 0) -- ++(0, -.05) ++(0, -.15) node [below, outer sep=0pt, inner sep=0pt] {\small\(1 \)};
\draw [decorate,decoration={brace,amplitude=10pt,mirror,raise=4pt},yshift=0pt]
(2.3-\x,0.5-\y) -- (2.3-\x,1.3-\y) node [black,midway,xshift=0.8cm] {\footnotesize
${B}_{\delta}$};
\draw [decorate,decoration={brace,amplitude=10pt,mirror,raise=4pt},yshift=0pt]
(2.5-\x,0.4-\y) -- (2.5-\x,1.4-\y) node [black,midway,xshift=0.8cm] {\footnotesize
${B}_{1}$};
\draw [dotted,thick,  domain=0:1, samples=40] plot ({0.2+1/5+\x^2*(1.15-1/5-0.2)},{0.5+0.55*\x});
\end{tikzpicture}
\caption{The cylinders $K_-,K_+$ and the spatial balls $\delta B,B$ in the proof of Lemma \ref{lem:weakl1log2} for $\tau = 1$, $(t_0,x_0) = 0$ and $r = 1$. The dotted line depicts the \textit{parabolic trajectory} connecting a point $(t,x) \in K_-$ to a point $(\eta,y)$ with $y \in B$. \label{fig:log2}}
\end{figure}
	
	We have
	\begin{align*}
		g(t,x)-c(u) &= \frac{1}{c_\varphi} \int_B (g(t,x)-g(\eta,y))\varphi^2(y) \dx y = -\frac{1}{c_\varphi} \int_B \int_0^1 \frac{\dx}{\dx r} g(\gamma(r)) \dx r \varphi^2(y) \dx y \\
		&= -\frac{1}{c_\varphi} \int_0^1\int_B \big( 2(\eta-t)r [\partial_t g](\gamma(r)) +  (y-x) \cdot [\nabla g](\gamma(r)) \big)  \varphi^2(y) \dx y \dx r \\
		&\le \frac{1}{c_\varphi}\int_0^1  \int_B \left(-2(\eta-t)r[\nabla \cdot (A \nabla g)](\gamma(r))-2(\eta-t)r[\langle{A \nabla g}, \nabla g \rangle](\gamma(r)) \vphantom{ -(y-x) \cdot [\nabla g](\gamma(r)) \vphantom{\abs{A \nabla g}^2}} \right. \\
		&\hphantom{=\frac{1}{c_\varphi}\int_0^1  \int_B \left( \right.}\left. -(y-x) \cdot [\nabla g](\gamma(r)) \vphantom{\abs{A \nabla g}^2}\right) \varphi^2(y)  \dx y \dx r,
	\end{align*}
	by the suprtsolution property of $g$.
	
	We start by performing a partial integration in the first term. We substitute $\tilde{y} =\Phi(y) =  \Phi_{r,t,x,\eta}(y) := \gamma_x(r)$, hence
	\begin{align*}
		&\int_0^1  \int_B -r[\nabla \cdot (A \nabla g)](\gamma(r)) \varphi^2(y) \dx y \dx r \\
		&= -\int_0^1  \int_{\Phi(B)} [\nabla \cdot (A \nabla g)](\gamma_t(r),\tilde{y}) \varphi^2\left(\frac{1}{r} \tilde{y} +\left(1-\frac{1}{r}\right)x\right) r^{-d+1}\dx \tilde{y} \dx r \\
		&= 2\int_0^1  \int_{\Phi(B)} (A \nabla g)(\gamma_t(r),\tilde{y}) \cdot [\nabla \varphi]\left(\frac{1}{r} \tilde{y} +\left(1-\frac{1}{r}\right)x\right) \varphi\left(\frac{1}{r} \tilde{y} +\left(1-\frac{1}{r}\right)x\right) r^{-d}\dx \tilde{y} \dx r \\
		&= 2\int_0^1  \int_B (A \nabla g)(\gamma(r)) \cdot [\nabla \varphi](y) \varphi(y)\dx {y} \dx r \\
		&\le \frac{4\sqrt{\Lambda}}{1-\delta}\int_0^1  \int_B \abs{\nabla g}_A(\gamma(r)) \varphi(y)\dx {y} \dx r,
	\end{align*}
	by the Cauchy-Schwarz inequality for $\abs{\xi}_A^2 := \abs{\xi}^2_{A(\gamma(r))} := \langle A(\gamma(r)) \xi, \xi \rangle$. 
	
	We continue by estimating
	\begin{align*}
		g(t,x)-c(u) &\le  \frac{1}{c_\varphi}\int_0^1  \int_B \left(-2(\eta-t)r[\nabla \cdot (A \nabla g)](\gamma(r))-(\eta-t)r\abs{\nabla g}_A^2(\gamma(r))\right) \varphi^2(y) \dx y \dx r \\
		&\hphantom{=}+ \frac{1}{c_\varphi}\int_0^1  \int_B \left(-(y-x) \cdot [\nabla g](\gamma(r))-(\eta-t)r\abs{\nabla g}_A^2(\gamma(r))\right) \varphi^2(y) \dx y \dx r \\
		&\le \frac{\eta-t}{c_\varphi}\int_0^1  \int_B \left(  \frac{8\sqrt{\Lambda}}{1-\delta} \abs{\nabla g}_A(\gamma(r))\varphi(y)-r\abs{\nabla g}^2_A(\gamma(r)) \varphi^2(y) \right) \dx y \dx r \\
		&\hphantom{=}+\frac{1}{c_\varphi}\int_0^1  \int_B \left(\frac{2}{\sqrt{\lambda}}\abs{\nabla g}_A(\gamma(r))\varphi(y) -r(\eta-t)\abs{\nabla g}^2_A(\gamma(r)) \varphi^2(y) \right)  \dx y \dx r.
	\end{align*}
	Consequently,
	\begin{align} 
		&\int_0^{\eta-\iota}\int_B (g(t,x)-c(u))_+ \dx x \dx t \label{eq:gmc} \\
		 &\le \frac{1}{c_\varphi} \int_0^{\eta-\iota} \hspace{-5pt} (\eta-t) \int_B  \int_B \int_0^1 \left(  \frac{8 \sqrt{\Lambda}}{1-\delta} \abs{\nabla g}_A(\gamma(r))\varphi(y)-r\abs{\nabla g}^2_A(\gamma(r)) \varphi^2(y) \right)_+ \hspace{-8pt} \dx r \dx y \dx x \dx t  \nonumber\\
		&\hphantom{=}+\frac{1}{c_\varphi} \int_0^{\eta-\iota} \hspace{-5pt} (\eta-t) \int_B  \int_B \int_0^1 \left(  \frac{4\lambda^{-1/2}}{(\eta-t)} \abs{\nabla g}_A(\gamma(r))\varphi(y)-r\abs{\nabla g}^2_A(\gamma(r)) \varphi^2(y) \right)_+ \hspace{-8pt} \dx r \dx y \dx x \dx t.  \nonumber
	\end{align}
	Let $M>0$. Then 
	\begin{align*}
		&\int_0^{\eta-\iota}\int_B  \int_B \int_0^1 \left( M \abs{\nabla g}_A(\gamma(r))\varphi(y)-r\abs{\nabla g}_A^2(\gamma(r)) \varphi^2(y) \right)_+ \dx r \dx y \dx x \dx t \\
		 &= \int_0^{\eta-\iota}\int_B  \int_B \int_{0}^{1/2} \left(  M \abs{\nabla g}_A(\gamma(r))\varphi(y)-r\abs{\nabla g}_A^2(\gamma(r)) \varphi^2(y) \right)_+ \dx r \dx y \dx x \dx t \\
		 &\hphantom{=}+\int_0^{\eta-\iota}\int_B  \int_B \int_{1/2}^1 \left(  M \abs{\nabla g}_A(\gamma(r))\varphi(y)-r\abs{\nabla g}_A^2(\gamma(r)) \varphi^2(y) \right)_+ \dx r \dx y \dx x \dx t =: I_1 +I_2. 
	\end{align*}
	For the first term we substitute $\tilde{x} = \Psi_{r,t,\eta,y}(x):= \gamma_x(r)$ and $\tilde{t} = t+r^2(\eta-t)$, write $p = \abs{\nabla g}_{A(\tilde{t},\tilde{x}) }(\tilde{t},\tilde{x}) \varphi(y)$ and estimate as follows
	\begin{align*}
		I_1 &= \int_B   \int_{0}^{1/2} \int_{r^2\eta}^{\eta+(r^2-1)\iota} \int_{\Psi(B)} \left(  M \abs{\nabla g}_A(\tilde{t},\tilde{x})\varphi(y)-r\abs{\nabla g}^2_A(\tilde{t},\tilde{x}) \varphi^2(y) \right)_+ \\
		&\hphantom{====\int_{0}^{1/2} \int_{r^2\eta}^{\eta+(r^2-1)\iota} \int_{\Psi(B)}} \cdot (1-r)^{-d}(1-r^2)^{-1} \dx \tilde{x} \dx \tilde{t} \dx r \dx y \\
		&\le C \int_B   \int_{0}^{1/2} \int_0^{\eta} \int_{B} \left(  M p-rp^2 \right)_+  \dx \tilde{x} \dx \tilde{t} \dx r \dx y \\
		&= C \int_B \int_0^{\eta} \int_{B}  \int_{0}^{1/2}  \left(  M p-rp^2 \right)_+ \dx r \dx \tilde{x} \dx \tilde{t} \dx y 
	\end{align*}
	as $\Psi(B) \subset B$ for some $C = C(d)$. Considering the inner integral with $m = \max \{1/2,M/p \}$
	\begin{equation*}
		\int_{0}^{1/2}  \left(  M p-rp^2 \right)_+ \dx r  = mMp-\frac{m^2}{2}p^2 \le \frac{M^2}{\sqrt{2}}
	\end{equation*}
	for all $p >0$. This shows the bound $I_1 \le C\frac{M^2}{\sqrt{2}}\eta \abs{B}^2$.
	
	Regarding $I_2$, we use the Cauchy-Schwarz inequality and obtain 
	\begin{align*}
		I_2 \le \int_0^{\eta-\iota}\int_B  \int_B \int_{1/2}^1 \frac{M^2}{r} \dx r \dx y  \dx x \dx t\le \log(2)\eta M^2 \abs{B}^2. 
	\end{align*}
	
	We apply this estimate to the integrals in \eqref{eq:gmc} to conclude
	\begin{align*}
		\int_0^{\eta-\iota}\int_B (g(t,x)-c(u))_+ \dx x \dx t &\le \frac{\abs{B}^2}{c_\varphi} C(d,\delta,\eta, \iota)(\Lambda+\lambda^{-1})  \\
		&\le C(d,\delta,\eta,\iota,\varphi)\mu \abs{B}.
	\end{align*}

	From here, the first estimate follows as 
	\begin{align*}
		s \abs{\{ (t,x) \in K_- \colon \log(u) - c(u)> s \} } &\le  \int_0^{\eta-\iota}\int_B (g(t,x)-c(u))_+ \dx x \dx t \\
		&\le C(d,\delta,\eta,\iota,\varphi)\mu \abs{B}.
	\end{align*}
	Now for the second estimate, we choose again $\gamma(r) = (t+r^2(\eta-t),x+r(y-x))$ and write 
	\begin{align*}
		c(u)-g(t,x) &= \frac{1}{c_\varphi} \int_B (g(\eta,y)-g(t,x))\varphi^2(y) \dx y = \frac{1}{c_\varphi} \int_B \int_0^1 \frac{\dx}{\dx r} g(\gamma(r)) \dx r \varphi^2(y) \dx y \\
		&=  \frac{1}{c_\varphi} \int_0^1\int_B \big( 2(\eta-t)r [\partial_t g](\gamma(t)) +  (y-x) \cdot [\nabla g](\gamma(r)) \big)  \varphi^2(y) \dx y \dx r.
	\end{align*}
	Note that $\eta-t<0$ for $(t,x) \in K_+$, whence we can employ the supersolution property and argue as above. 
	
	If $A$ is not symmetric, then $\abs{\cdot}_A$ does not satisfy the Cauchy-Schwarz inequality in general. Thus we estimate a little bit coarser and end up with $(\Lambda+1/\lambda)^2$ in the end. 
	
	We used formal calculations. The inequality for $g(t,x)-c(u)$ needs to be understood in a weak sense, i.e. the trajectorial argument needs to be performed on the level of test functions. Note that one can also assume that the coefficients and the supersolution itself are $C^\infty$ by an approximation argument due to Aronson, \cite{aronson_uniqueness_1965,moser_pointwise_1971}. 
\end{proof}

\begin{proof}[Alternative proof of Theorem \ref{thm:harnack}]
Again we have to apply the lemma of Bombieri and Giusti twice on suitable cylinders. By adjusting the parameters suitably, we may assume $(-\frac{3}{4}\tau(1-\delta),2\tau) \times B_\beta \subset \Omega_T$ for some $\beta>1$. Choosing the cylinders as depicted in Figure \ref{fig:harnack2}, the proof follows along the lines of the proof given above.  Indeed set $U_\sigma = (2\tau - \frac{3}{4}\tau\sigma,2\tau) \times B_\sigma(0) $ and $ \bar{U}_\sigma = ( -\frac{3}{4}\tau(\sigma-\delta), \frac{3}{4}\tau \sigma) \times B_\sigma(0) $ for $0<\delta \le \sigma\le 1$. Set $K_- = \bar{U}_1$, $Q_- = \bar{U}_\delta$, $K_+ = U_1$ and $Q_+ = U_\delta$. The difference is that the cylinders $\bar{U}_\sigma$ and $U_\sigma$ are separated by a positive gap in time of size $\tau /2$. Now, with a slight abuse of notation, we apply Lemma \ref{lem:weakl1log2} with  $t_0 = -\frac{3}{4}\tau(1-\delta)$, $x_0 = 0$, $\tau = [2\tau+\frac{3}{4}\tau(1-\delta)]\frac{1}{\beta^2} = \frac{3}{4}\tau(\frac{11}{3}-\delta) \frac{1}{\beta^2} $, $\eta = \frac{\frac{7}{3}-\delta}{\frac{11}{3}-\delta}$, $\iota = \frac{1}{11-3\delta}$, $r= \beta$ and $\delta = 1/\beta$ to obtain the weak $L^1$-estimate on $K_-$ and $K_+$. Combined with Lemma \ref{lem:lplinf} and Lemma \ref{lem:bomb} this yields the Harnack inequality.

For the Harnack inequality, a gap between $Q_-$ and $Q_+$ is necessary. Hence this does not change the statement of the Harnack inequality. However, the dependence on $\delta,\tau$ might be worse. \qedhere

\begin{figure}[H]
\centering
\tikzmath{\x = -0.2; \y = 0.2;}
\begin{tikzpicture}[scale=4]
  \draw[thick,->] (0, 0) -- (2.6,0) node[right] {$t$};
  \draw[thick,->] (0, 0) -- (0,0.3) node[left] {$x$};
  \draw[draw=black] (0.0-\x,0.3-\y) rectangle ++(2.3,1.2);
  \draw[draw=black] (0.0-\x,0.4-\y) rectangle ++(2.3,1);
  \draw[draw=black] (1.9-\x,0.7-\y) rectangle ++(0.4,0.4);
  \draw (1.9-\x,1.1-\y) node[anchor=north west] {$Q_+$};
  \draw[draw=black] (0.4-\x,0.7-\y) rectangle ++(0.3,0.4);
  \draw (0.7-\x,1.1-\y) node[anchor=north east] {$Q_-$};
  \draw[draw=black] (0.0-\x,0.4-\y) rectangle ++(1.1,1);
  \draw[draw=black] (1.4-\x,0.4-\y) rectangle ++(0.9,1);
  \draw[draw=black] (0.2-\x,0.5-\y) rectangle ++(0.7,0.8);
  \draw[draw=black] (1.6-\x,0.5-\y) rectangle ++(0.7,0.8);
  \draw (1.1-\x,1.4-\y) node[anchor=north east] {$K_-$};
  \draw (1.4-\x,1.4-\y) node[anchor=north west] {$K_+$}; 
  \draw (0.9-\x,1.3-\y) node[anchor=north east] {$\bar{U}_\sigma$};
  \draw (1.6-\x,1.3-\y) node[anchor=north west] {$U_\sigma$};
  \draw [thick] (0.0-\x, 0) -- ++(0, -.05) ++(0, -.15) node [below, outer sep=0pt, inner sep=0pt] {\small\(-\frac{3}{4}\tau(1-\delta)\)};
  \draw [thick] (0.4-\x, 0) -- ++(0, -.05) ++(0, -.15) node [below, outer sep=0pt, inner sep=0pt] {\vphantom{\small\(-\frac{3}{4}\tau(1-\delta)\)}\small\(0\)};
   \draw [thick] (0.7-\x, 0) -- ++(0, -.05) ++(0, -.15) node [below, outer sep=0pt, inner sep=0pt] {\small\(\frac{3}{4}\tau\delta  \)};
    \draw [thick] (1.1-\x, 0) -- ++(0, -.05) ++(0, -.15) node [below, outer sep=0pt, inner sep=0pt] {\small\(\frac{3}{4}\tau \)\vphantom{\small\(-(1-\delta)(\tau-\iota)\)}};
    \draw [thick] (1.4-\x, 0) -- ++(0, -.05) ++(0, -.15) node [below, outer sep=0pt, inner sep=0pt] {\small\( \frac{5}{4}\tau \)\vphantom{\small\(-(1-\delta)(\tau-\iota)\)}};
      \draw [thick] (1.9-\x, 0) -- ++(0, -.05) ++(0, -.15) node [below, outer sep=0pt, inner sep=0pt] {\small\( 2\tau -\frac{3}{4}\tau\delta  \)};
      \draw [thick] (2.3-\x, 0) -- ++(0, -.05) ++(0, -.15) node [below, outer sep=0pt, inner sep=0pt] {\small\(2 \tau \)\vphantom{\small\(-(1-\delta)(\tau-\iota)\)}};
  \draw [decorate,decoration={brace,amplitude=10pt,mirror,raise=4pt},yshift=0pt]
(2.3-\x,0.7-\y) -- (2.3-\x,1.1-\y) node [black,midway,xshift=0.8cm] {\footnotesize
${B}_{\delta}$};
\draw [decorate,decoration={brace,amplitude=10pt,mirror,raise=4pt},yshift=0pt]
(2.5-\x,0.5-\y) -- (2.5-\x,1.3-\y) node [black,midway,xshift=0.8cm] {\footnotesize
${B}_{\sigma}$};
\draw [decorate,decoration={brace,amplitude=10pt,mirror,raise=4pt},yshift=0pt]
(2.7-\x,0.4-\y) -- (2.7-\x,1.4-\y) node [black,midway,xshift=0.8cm] {\footnotesize
${B}_{1}$};
\draw [decorate,decoration={brace,amplitude=10pt,mirror,raise=4pt},yshift=0pt]
(2.9-\x,0.3-\y) -- (2.9-\x,1.5-\y) node [black,midway,xshift=0.8cm] {\footnotesize
${B}_\beta$};
\end{tikzpicture}
\caption{The cylinders $Q_-,Q_+,\bar{U}_\sigma,U_\sigma,K_-,K_+$ and the spatial balls $B_\delta,B_\sigma,B_1,{B}_\beta$ in the proof of Theorem \ref{thm:harnack} for $(t_0,x_0) = 0$ and $r = 1$. \label{fig:harnack2}}
\end{figure}
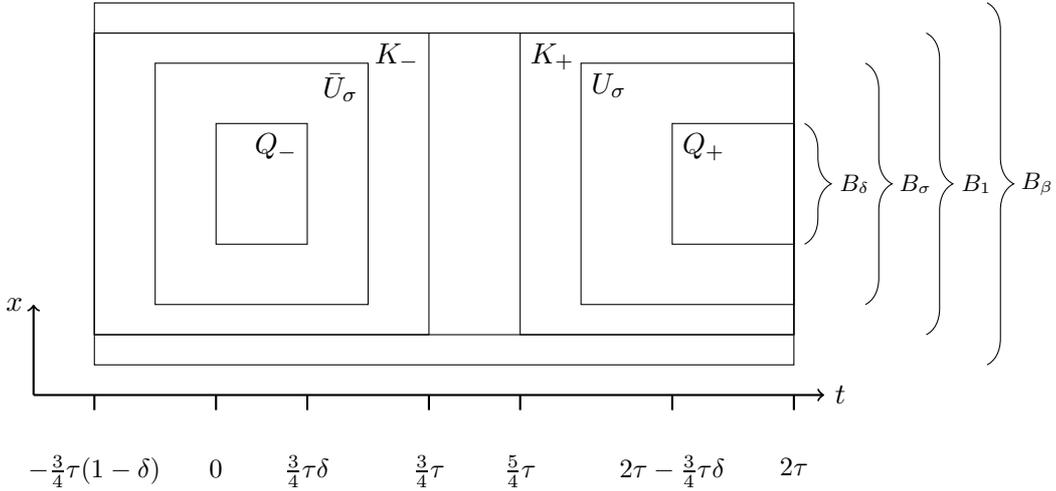
\end{proof}

\section{Elliptic equations}
The trajectorial approach works in the elliptic case, too.
Let $\Omega \subset \R^d$ be an open subset. We consider weak solutions $u \in H^1(\Omega)$ to the equation 
\begin{equation} \label{eq:elliptic}
	-\nabla \cdot (A \nabla u) = 0,
\end{equation}
where $A \in L^\infty(\Omega;\R^{d \times d})$ such that $\lambda \abs{\xi} \le \langle A(x) \xi ,\xi \rangle $ and $\sum_{i,j = 1}^d \abs{A_{ij}(x)}^2 \le \Lambda^2$ for almost all $x \in \Omega$, any $\xi \in \R^d$ and some $0<\lambda<\Lambda $. We write $\mu = \sqrt{\frac{\Lambda}{\lambda}}$ if $A$ is symmetric and $\mu = \frac{\Lambda}{\lambda}$ else.

Using the same method, we can prove the following weak $L^1$-estimate for the logarithm of supersolution. See also \cite{clement_apriori_2004,saloff_aspects_2001} for a proof using the Poincar\'e inequality. 

\begin{lemma}
	Let $\delta \in (0,1)$, $\epsilon>0$ and $u \in H^1(\Omega)$, $u \ge \epsilon>0$ be a weak supersolution to \eqref{eq:elliptic} in $\Omega$. Then, for any ball $B = B_r(x_0) \subset \Omega$ with $r>0$ there exists $C = C(d,\delta)>0$ and $c = c(u)$ such that 
	\begin{equation*}
		\abs{B_{\delta r}(x_0) \cap \{ \log(u) > c+ s \} } \le C \mu \abs{B} s^{-1}
	\end{equation*}
	and
	\begin{equation*}
		\abs{B_{\delta r}(x_0) \cap \{  \log(u) < c-s \} } \le C\mu \abs{B} s^{-1}.
	\end{equation*}
\end{lemma} 

\begin{proof}
	Following the same ansatz as in the proof of Lemma \ref{lem:weakl1log2}, we use the trajectory $\gamma(r) = x+r(y-x)$. The supersolution property is used by adding the following nonnegative term 
	\begin{equation*}
		0 \le -\alpha r \ \nabla \cdot (A \nabla g) - \alpha r \ \langle A \nabla g, \nabla g \rangle,
	\end{equation*}
	where $g = \log(u)$ and $\alpha >0 $.  From here, we proceed as above. Assuming that $A$ is symmetric this leads to a factor of the form $\alpha \Lambda + \frac{1}{\alpha \lambda} $ on the right-hand side. Optimising the constant with respect to the choice of $\alpha$, we end up with the factor $\mu = \sqrt{\frac{\Lambda}{\lambda}} $ on the right-hand side. In the nonsymmetric case we obtain an estimate with the factor $\mu = {\frac{\Lambda}{\lambda}} $.
\end{proof}

Together with the elliptic $L^p-L^\infty$ estimate and the Lemma of Bombieri and Giusti, this can be used to prove the Harnack estimate with the optimal dependency of the Harnack constant on $\lambda,\Lambda$. We refer to \cite{clement_apriori_2004,saloff_aspects_2001} for more details and to \cite{mosconi_optimal_2018,moser_pointwise_1971} for a discussion of the optimality.

\section{Further comments}
A trajectorial interpretation of De Giorgi's proof of the elliptic Harnack inequality is due to Vasseur and can be found in \cite{vas_degiorgi_2016}. 

Based on Vasseur's argument, De Giorgi's ideas have been transferred to kinetic equations in \cite{golse_harnack_2016,guerand_quant_2021}. We aim to transfer the trajectorial interpretation of Moser's proof to kinetic equations or even general hypoelliptic equations in the future. We hope to combine an explicit form of the Harnack constant with the framework of \cite{niebel_kinetic_nodate-1} to study the global existence of solutions to nonlinear kinetic equations. The authors employed \textit{kinetic trajectories} in \cite{niebel_poincare_2022} to prove a weak Poincar\'e inequality for weak subsolutions to the Kolmogorov equation. 

The proof of the logarithmic estimates gives some insights into the connection to the proof of the Harnack inequality due to Li and Yau. They prove a differential estimate, which, combined with similar trajectories, gives the Harnack estimate for classical solutions of the heat equation, \cite{li_parabolic_1986}. The Li-Yau estimate is a strong pointwise a priori estimate and is interesting in many different settings, for example, in the context of analysis on graphs, see \cite{dier_discrete_2021} and references therein.  

Moser's approach based on the lemma of Bombieri and Giusti can also be used to prove a weak Harnack inequality with optimal range for the exponent. We refer to \cite{clement_apriori_2004,saloff_aspects_2001,zacher_harnack_2013} for more details. 

\bibliographystyle{amsplain}

\providecommand{\bysame}{\leavevmode\hbox to3em{\hrulefill}\thinspace}
\providecommand{\MR}{\relax\ifhmode\unskip\space\fi MR }

\providecommand{\MRhref}[2]{%
  \href{http://www.ams.org/mathscinet-getitem?mr=#1}{#2}
}
\providecommand{\href}[2]{#2}

$ $\\

\end{document}